\documentclass[journal]{IEEEtran}
\ifCLASSINFOpdf
\else
\fi
\usepackage{amsmath,amssymb,amsthm}
\usepackage{graphicx,subfigure,epstopdf}
\usepackage[comma,numbers,square,sort&compress]{natbib}
\usepackage{epstopdf}
\usepackage{hyperref}
\usepackage{subfigure}
\usepackage{graphicx}
\usepackage{epstopdf}
\usepackage{epsfig}
\usepackage{enumerate}
\usepackage[utf8]{inputenc}
\usepackage[english]{babel}
\usepackage{fancyhdr}
\usepackage{color}
\usepackage{dsfont}
\usepackage{changes}
\usepackage{tikz}
\usepackage{todonotes}

\usetikzlibrary{positioning}

\newcommand{\EQ}{\begin{eqnarray}}

\newcommand{\EN}{\end{eqnarray}}

\newcommand{\EQQ}{\begin{eqnarray*}}

\newcommand{\ENN}{\end{eqnarray*}}

\newcommand{\col}{\mbox{col}}
\newcommand{\row}{\mbox{row}}
\newcommand{\vecc}{\mbox{vec}}

\newtheorem{thm}{\bf \em{Theorem}}

\newtheorem{lem}{\bf \em{Lemma}}
\newtheorem{rem}{\bf \em{Remark}}

\newtheorem{prob}{\bf \em{Problem}}

\newtheorem{ass}{\bf \em{Assumption}}

\newcommand{\myr}{\color{black}}
\newcommand{\hwz}{\color{black}} 
\newcommand{\hw}{\color{black}} 
\newcommand{\Mathcal}{\mathcal}
\usepackage{changes}

\allowdisplaybreaks
\begin{document}
\onecolumn
{\Large

This paper was originally published in

Wang, Shimin, Hongwei Zhang, Simone Baldi, and Renxin Zhong. "Leaderless Consensus of Heterogeneous Multiple Euler-Lagrange Systems with Unknown Disturbance." IEEE Transactions on Automatic Control, DOI: 10.1109/TAC.2022.3172594, 2022.

We have spotted some typos in Remark 1, Lemma 3 and Remark 3
where, in some place,  the $\mathds{R}$ should be $\mathds{C}$


These typos are now corrected and the modified version is posted here.}


%

 \twocolumn
\title{Leaderless Consensus of Heterogeneous Multiple Euler-Lagrange Systems with Unknown Disturbance}

\author{Shimin~Wang,~Hongwei~Zhang, Simone~Baldi and Renxin~Zhong
\thanks{This work was supported in part by NSERC, by the National Natural Science Foundation of China under grant 61773322, by the Research Fund for International Scientists under grant 62150610499, and by the Key Intergovernmental Special Fund of National Key Research and Development Program under grant 2021YFE0198700. (Corresponding authors: Hongwei
Zhang and Simone Baldi)}
\thanks{Shimin~Wang is with the Department of Chemical Engineering, Queen’s University, Kingston, ON K7L 3N6, Canada. E-mail: shimin.wang@queensu.ca.}
\thanks{Hongwei Zhang is with the School of Mechanical Engineering and Automation, Harbin Institute of Technology, Shenzhen, 518055, P. R. China. E-mail: hwzhang@hit.edu.cn.}
\thanks{ Simone Baldi is with the School of Mathematics, Southeast University, Nanjing, P. R. China, and with the Delft Center for Systems and Control, Delft University of Technology, 2628 CD Delft, Netherlands. E-mail: s.baldi@tudelft.nl.}
\thanks{Renxin Zhong is with the Guangdong Key Laboratory of Intelligent Transportation Systems, School of Intelligent Systems Engineering, Sun Yat-Sen University, Guangzhou 510275, P. R. China. E-mail: zhrenxin@mail.sysu.edu.cn.}
}
\maketitle

\begin{abstract}
This paper studies the leaderless consensus problem of {heterogeneous} multiple networked Euler-Lagrange systems subject to persistent disturbances with unknown constant biases, amplitudes, initial phases and frequencies. The main characteristic of this study is that none of the agents has information of a common reference model or of a common reference trajectory. Therefore, the agents must simultaneously and in a distributed way: achieve consensus to a common reference model (group model); achieve consensus to a common reference trajectory; {and} reject the unknown disturbances. We show that this is possible via a suitable combination of techniques of distributed `observers', internal model principle and adaptive regulation. The proposed design generalizes recent results on group model learning, which have been studied for linear agents over undirected networks. In this work, group model learning is achieved for Euler-Lagrange dynamics over directed networks in the presence of persistent unknown disturbances.
\end{abstract}


\section{Introduction}
Euler-Lagrange (EL) systems have found widespread applications in engineering and can model a variety of mechanical systems, such as marine vessels \cite{peng2016containment}, rigid {spacecrafts} \cite{chen2014attitude}, and robot manipulators \cite{lewis2004Book-robot,zarikian2005external}. {Since} precise modeling of an EL system is very difficult {in practice and} disturbances are always entangled with the system movement, control of uncertain EL systems with disturbance rejection has been an important issue in control community \cite{Patre2011TAC-EL, chen2009attitude, lu2019adaptive}. A recent work \cite{lu2019adaptive} solved a global asymptotic tracking control problem of EL systems with disturbance rejection, where the disturbance is a combination of sinusoidal signals with unknown frequencies, amplitudes and phase angles. However, a similar problem becomes more challenging in a cooperative setting with multiple EL systems since, in addition to rejecting disturbances, the systems should achieve a common behavior with limited information (only using local information from a few neighbors).

Cooperative control of multiple EL systems has been intensively investigated in the past two decades mainly under two formulations, i.e., leader-following consensus {\myr(with a single leader or multiple leaders)} \cite{nuno2018consensus, cai2016leader, klotz2016robust} and leaderless consensus \cite{ren2009distributed, klotz2014asymptotic, nuno2020strict, MeiJ2020TCNS-EL}. For leader-following consensus,
{\hwz a leader (or a group of leaders) generate a desired trajectory (or a convex hull) that all follower agents should follow. The desired trajectories can be  time-varying  and the tracking problem} will become even more stringent if there exist some external disturbances \cite{liu2019leader}. In this sense, the tracking control of a single Euler-Lagrange system as in \cite{lu2019adaptive} can be viewed as a special case of the leader-following consensus with one leader (i.e. the desired trajectory) and one follower. To {tackle} the local information challenge, the idea of using a distributed observer \cite{cai2014leader} or an adaptive distributed observer \cite{cai2016leader,wang2018adaptive} was proposed for leader-following consensus.  The idea is that only part of the follower {agents} can directly get access to the state and system matrix information of the leader, {while} the rest of the follower {agents} {should} estimate the leader's information using observers.

In many practical scenarios, there is no such leader.
For example, when a leaderless swarm of UAVs performs surveillance missions, individuals need to reach a consensus in altitude and heading angle and must coordinate with each other a commonly agreed trajectory to track \cite{Rao2014TCST-sliding}. A similar setting has been reported for a group of robotic arms equipped on different mobile robots to cooperatively scan a target area \cite{ren2009distributed}.
Most existing works on leaderless consensus
of multiple networked EL systems typically allow the common
trajectory to be time-invariant \cite{ren2009distributed, MeiJ2020TCNS-EL}. Even when a disturbance is considered, as in \cite{MeiJ2020TCNS-EL}, it is assumed that the final consensus equilibrium is a constant trajectory.
%

As synchronization of uncertain heterogeneous multi-agent systems to more complex trajectories  requires either a leader agent generating a desired trajectory, or a common model according to the internal model principle, 
it is interesting to ask: what can be done without a leader? This problem has not been sufficiently investigated until very recently \cite{baldi2019leaderless, YanYM2020TCNS-autonomous, chenzy2020}. The work \cite{baldi2019leaderless} gave a first answer for a special class of linear multi-agent systems, i.e., heterogeneous oscillator systems. It formulates leaderless consensus as a `virtual' leader-following consensus problem. It shows that there exists a `group model' that has the same structure as the oscillators. Via consensus dynamics, each agent can learn the parameters of the group model without its direct knowledge, and finally synchronize to it. In this sense, synchronization of multiple oscillators to a non-constant trajectory is achieved. More recently, a similar framework has been proposed in \cite{chenzy2020} for leaderless consensus of linear time-varying multi-agent system, whereas \cite{YanYM2020TCNS-autonomous} proposes a two-step approach, i.e., dynamics synchronization and state synchronization, and provides sufficient conditions for the efficacy of this two-step design. However, \cite{baldi2019leaderless, YanYM2020TCNS-autonomous,chenzy2020} only consider linear dynamics or undirected communication graphs.

Motivated by these recent achievements, this paper aims to solve a leaderless consensus problem of uncertain heterogeneous EL systems with unknown disturbances over directed graphs. The disturbance is a compound sinusoidal signal with unknown magnitudes, frequencies and phase angles. Each agent aims to achieve consensus to a complex time-varying trajectory, cooperatively contributed by the whole group of agents. This include the constant consensus equilibrium \cite{MeiJ2020TCNS-EL} as a special case. Therefore, the agents must simultaneously and in a distributed way: achieve consensus to a common group system matrix; achieve consensus to a common reference trajectory; and reject the unknown harmonic disturbances. Inspired by both \cite{chen2009attitude,cai2016leader,lu2019adaptive} and \cite{baldi2019leaderless}, we show that this is possible via a suitable combination of consensus dynamics, internal model principle and adaptive regulation. More specifically, we propose an `observer' for each agent, whose task is to `observe' the state and system matrix of {\myr an autonomous system which is not pre-specified but arising from the inherent properties and the initial states of the agents. We put the term `observe' in quotes since this {\myr autonomous system} does not exist a priori.} In other words, it is an imaginary one, and is generated through the collaboration of all observers of the group of agents.

The contribution and novelties of our approach are summarized as follows:
\begin{enumerate}
  \item  In place of considering linear dynamics and undirected graphs, we solve a leaderless consensus problem of uncertain heterogeneous EL systems with unknown disturbances over directed graphs. This requires to develop new technical results (Lemmas 2-4 in this work) not reported in the literature.
  \item  Based on the consensus stage, we design a cooperative
controller for each EL system to synchronize to the
observer while rejecting in an adaptive way the external unknown disturbances. 
  \item  Instead of a bounded tracking signal as in the single Euler-Lagrange system case \cite{lu2019adaptive}, we only require that the derivative of the final consensus state is bounded without imposing bounds on the cooperatively agreed trajectory.
\end{enumerate}

The rest of this paper is organized as follows. The problem is formulated in Section \ref{section2}. In Section \ref{section3}, distributed `observers' are designed for all agents, which collaboratively generate {\myr an autonomous system which is not pre-specified but arising from the inherent properties and the initial states of the agents}. The main result is presented in Section \ref{section4}, followed by a numerical example in Section \ref{section5}. Section \ref{section6} concludes the paper.

\textbf{Notation:} Notation $\|\cdot\|$ is the Euclidean norm. The set of (positive) real numbers are denoted by ($\mathds{R}_{+}$) $\mathds{R}$. The set of complex numbers are denoted by $\mathds{C}$. For $X_i\in \mathds{R}^{n_i\times m}$, $i=1,\dots,N$, let $\col(X_1,\dots,X_N)=[X_{1}^{T},\dots,X_N^T]^T$ and $\mathds{1}_N=\col(1,\dots,1)\in \mathds{R}^{N}$. For $X_i\in \mathds{R}^{m\times n_i}$, $i=1,\dots,N$, let $\row(X_1,\dots,X_N)=[X_{1},\dots,X_N]$. For any matrix $X\in \mathds{R}^{m\times n}$, let $\vecc\left(X\right)=\col\left(X_1,\dots,X_n\right)$, where $X_i\in \mathds{R}^{m}$ denotes the $i$th column of $X$. Finally, $\otimes$ denotes the Kronecker product, and  $\circ$ denotes the Tracy-Singh product.

\section{Problem Formulation}\label{section2}
Consider $N$ agents represented by the following Euler-Lagrange dynamics
\begin{equation}\label{MARINEVESSEL1}
  \mathcal{M}_i\left(q_i\right)\ddot{q}_i+\mathcal{C}_i\left(q_i,\dot{q}_i\right)\dot{q}_i+ {G_i}\left(q_i\right)=\tau_i+  d_i
\end{equation}
where, for each agent $i$, $q_i\in \mathds{R}^n$ is the vector of generalized coordinates,
$\Mathcal{M}_i\left(q_i\right)\in \mathds{R}^{n\times n}$ is the symmetric positive definite inertia matrix, $\Mathcal{C}_i\left(q_i,\dot{q}_i\right)\dot{q}_i\in \mathds{R}^{n}$ is the vector of Coriolis and centripetal forces, ${G_i}\left(q_i\right)\in \mathds{R}^{n}$ is the vector of gravitational force, $\tau_i\in \mathds{R}^{n}$ is the control torque, and $d_i=\col\left(d_{i1},\dots,d_{in}\right)\in \mathds{R}^n$ is the external disturbance, taking the form
\begin{align}\label{exdistur}
d_{is}(t)&=\psi_{is,0}+\sum\nolimits_{k=1}^{n_{is}}\psi_{is,k}\sin(\sigma_{is,k}t+\phi_{is,k}),\\
& \qquad i=1,\dots,N, \ s=1,\dots,n{,} \nonumber
\end{align}
where $\psi_{is,0}, \phi_{is,k}\in \mathds{R}$, $\psi_{is,k}, \sigma_{is,k}\in \mathds{R}_{+}$ are constant biases, initial phases, amplitudes, and frequencies. Biases, initial phases, amplitudes, and frequencies can all be arbitrary and unknown.
In line with most Euler-Lagrange literature \cite{lewis2004Book-robot}, let the dynamics (\ref{MARINEVESSEL1}) satisfy the following properties:
\begin{enumerate}[(1)]
  \item \label{propbound}The inertia matrix $\Mathcal{M}_i\left(q_i\right)$ is symmetric and uniformly positive definite such that $k_{\underline{m}} I \leq \Mathcal{M}_i\left(q_i\right)\leq k_{\overline{m}} I$ {\myr for some
positive scalars $k_{\underline{m}}$ and $k_{\overline{m}}$.} Also, $\|\mathcal{C}_i\left(q_i,\dot{q}_i\right)\|\leq k_c\|\dot{q}_i\|$, and $\|{G_i}\left(q_i\right)\|\leq k_g$ for some positive scalars $k_c$ and $k_g$.
  \item \label{prop1} For all $x,y\in \mathds{R}^n$, $\Mathcal{M}_i\left(q_i\right)x+\Mathcal{C}_i\left(q_i,\dot{q}_i\right)y+G_i\left(q_i\right)=Y_i\left(q_i,\dot{q}_i,x,y\right)\Theta_i$, where $Y_i\left(q_i,\dot{q}_i,x,y\right)\in \mathds{R}^{n\times q}$ is a known regression matrix and $\Theta_i\in \mathds{R}^{q}$ is a constant vector consisting of the uncertain parameters of (\ref{MARINEVESSEL1}).
  \item$\dot{\mathcal{M}}_i\left(q_i\right)-2\Mathcal{C}_i\left(q_i,\dot{q}_i\right)$ is skew symmetric, $\forall q_i,\dot{q}_i\in \mathds{R}^n$.
\end{enumerate}

Let the agents \eqref{MARINEVESSEL1} interact according to a static directed graph $\mathcal{G}=\{\mathcal{V}, \mathcal{E}, \mathcal{A}\}$ where the vertex set is $\mathcal{V}=\{1,2,\dots, N\}$, and the edge set is $\mathcal{E}\subseteq \mathcal{V}\times \mathcal{V}$. We use $\mathcal{A}=\left[a_{ij}\right]\in \mathbb{R}^{ N \times N}$ to denote the adjacency matrix of graph $\mathcal{G}$, where $a_{ij}>0$ if $\left(j,i\right)\in \mathcal{E}$, and $a_{ij}=0$ otherwise. Let $\mathcal{L} \in \mathds{R}^{N\times N}$ be the Laplacian matrix of graph $\mathcal{G}$, and $\mathcal{N}_i=\{j|(j,i)\in  \mathcal{E}\}$
be the neighbor set of agent $i$. For more details on graph theory, readers are referred to \cite{lewis2014cooperative}.
The following property holds for the Laplacian matrix $\mathcal{L}$:
\begin{lem}\cite{ren2005consensus}\label{WRWB2005} If the communication graph $\mathcal{G}$ contains a spanning tree, then $0$ is a simple eigenvalue of the Laplacian matrix $\mathcal{L}$, and all the other $N - 1$ eigenvalues have positive real parts.
\end{lem}

\begin{prob}[Leaderless Consensus Problem]\label{ldlesp}Consider the networked Euler-Lagrange systems (\ref{MARINEVESSEL1}) with communication graph $\mathcal{G}$. Find a distributed control law
such that, for any external
disturbance with arbitrary $\psi_{is,0}$, $\psi_{is,k}$, $\phi_{is,k}$ and $\sigma_{is,k}$ as in \eqref{exdistur}, and arbitrary initial conditions $q_i(0)$ and $\dot{q}_i(0)$,
the trajectories $q_i(t)$ and $\dot{q}_i(t)$ exist and are bounded for all $t\geq0$, and the following consensus results are achieved,
$$\lim\limits_{t\rightarrow\infty}\left(q_i\left(t\right)-q_j\left(t\right)\right)=0,~~\lim\limits_{t\rightarrow\infty}\left(\dot{q}_i\left(t\right)-\dot{q}_j\left(t\right)\right)=0, \forall i, j.$$
\end{prob}

To solve Problem \ref{ldlesp}, we need the following assumption, which is a standard assumption for directed static communication graphs \cite{ren2005consensus}.
\begin{ass}\label{ass2}
The communication graph $\mathcal{G}$ contains a spanning tree.
\end{ass}

\begin{rem}\label{rem1} {\myr Under Assumption 1, for} the Laplacian matrix $\mathcal{L}\in \mathds{R}^{N\times N}$ of the communication graph $\mathcal{G}$,
 there exists a nonsingular matrix $U\in \mathds{C}^{N\times N}$ such that $U^{-1}\mathcal{L}U=J_{\mathcal{L}}$, where $J_{\mathcal{L}}$ is the Jordan canonical form of $\mathcal{L}$. In the following, let us denote $\lambda_{1}$ as the nonzero minimum real part among the eigenvalues of $\mathcal{L}$.
\end{rem}
\section{{Distributed observer and dynamic compensator}}\label{section3}

In this section, a distributed observer is designed for each agent so that all these observers will achieve consensus to  {\myr an autonomous system determined by the inherent properties and the initial states of the agents}. Additionally,  an internal model based dynamic compensator is designed to deal with the uncertain disturbances.

\subsection{{Design of a distributed observer}}
We propose a distributed observer for each agent as follows:
 \begin{subequations}\label{sweq1}
\begin{align}
\dot{S}_i&= \mu_1\sum\nolimits_{j\in \mathcal{N}_i}{a_{ij}(S_j-S_i)}\label{sweq1i}\\
\dot{\eta}_i &=S_i\eta_i + \mu_2\sum\nolimits_{j\in \mathcal{N}_i} {a_{ij}(\eta_j-\eta_i)}\label{sweq1ii}
 \end{align}
\end{subequations}
where $S_i\in \mathds{R} ^{n\times n}$ and $\eta_i\in \mathds{R} ^{n}$ {\myr are the estimated system matrix and state of the autonomous system, respectively}.

The main difference between \eqref{sweq1} and other adaptive distributed observers in the literature, e.g. \cite{cai2016leader,cai2017adaptive} is that the adaptive distributed observers in \cite{cai2016leader,cai2017adaptive} require an explicit leader agent, generating an a priori
reference trajectory for the network, while \eqref{sweq1} requires no leader agent and all agents works cooperatively to construct {\hw an autonomous system}.

%
In the following development, {\hw we shall show how to construct  an autonomous system} by the proposed observer \eqref{sweq1}. To this purpose, a technical lemma is needed.

\begin{lem}\label{lema1}Consider the system
\begin{align}\dot{x}=F(t)x, \label{vFx}\end{align}
where $x\in \mathds{R}^{n}$, and $F(\cdot): \mathds{R}\rightarrow \mathds{R}^{n\times n}$ is bounded and piecewise continuous for all $t\geq0$. If $F(t)$ {vanishes} exponentially, then {$x$ converges} to a bounded vector.
\end{lem}
\begin{proof} Since $F(t)$ vanishes exponentially, there exist positive {constants} $\alpha$ and $\lambda$, such that $\|F(t)\|\leq \alpha e^{-\lambda t}$. Let $V=x^Tx$. The time derivative of $V$ along the system \eqref{vFx} is
\begin{align*}
  \dot{V}
  &=x^T\left(F(t)+F^T(t)\right)x \nonumber\\
   &\leq2\alpha e^{-\lambda t}V.
\end{align*}
Then, $\forall t \geq 0$,
\begin{align*}
V(t)&\leq e^{\int_{0}^{t}2\alpha e^{-\lambda \tau}d\tau}V\left(0\right)\nonumber\\
&\leq e^{\frac{2\alpha}{\lambda}}\|x\left(0\right)\|^2,
\end{align*}
which implies that $\|x(t)\|$ is bounded {for all} $x(0)$ and $t\geq 0$. Hence, for system \eqref{vFx}, $F(t)x$ will converge to zero exponentially at the rate of $\lambda$. Clearly, there exists {an} $x^*\in \mathds{R}^{n}$ such that {$\lim_{t\rightarrow \infty}x \left(t\right)= x^*$ exponentially} at the rate of $\lambda$.
%
\end{proof}

\begin{rem}\myr A related result is reported in Lemma 1 of \cite{cai2017adaptive}. However, Lemma 1 in \cite{cai2017adaptive} considers the system $\dot{x}=F_0x+F(t)x$, where matrix $F_0$ needs to be Hurwitz, proving that $x$ converges to zero.  Clearly, the system \eqref{vFx} in the proposed Lemma \ref{lema1} cannot be covered by \cite{cai2017adaptive}, due to the absence of the Hurwitz matrix $F_0$.
\end{rem}
Now we are ready to show the consensus of dynamics \eqref{sweq1}.

\begin{lem}\label{lema2}
Consider the dynamics (\ref{sweq1i}). Under Assumption \ref{ass2}, for any positive $\mu_1$ and any initial $S_i\left(0\right)$, the matrix signals $S_i(t)$ will achieve consensus exponentially, for $i=1,\dots,N$.
\end{lem}
\begin{proof}
For {notational conciseness}, define $\bar{S}=\col\left(S_{1},\dots,S_{N}\right)$. Then, we can rewrite the dynamics (\ref{sweq1i}) {in a compact way}
 \begin{align}
 \dot{\bar{S}}=&-\mu_1\left(\mathcal{L}\otimes I_n\right)\bar{S}.\label{eS1a}
 \end{align}
 By Remark \ref{rem1}, let $\Phi=\left(U^{-1}\otimes I_n\right)\bar{S} \in \mathds{C}^{Nn\times n}$. Then, equation {\eqref{eS1a}} can be rewritten as
  \begin{align}\label{eS2}
  \dot{\Phi}=-\mu_1\left(J_\mathcal{L}\otimes I_n\right)\Phi,
  \end{align}
where $J_{\mathcal{L}}$ is the Jordan canonical form of $\mathcal{L}$. Since the graph $\mathcal{G}$ contains a spanning tree,   we have, from Lemma $\ref{WRWB2005}$, that $0$ is a simple eigenvalue of $J_{\mathcal{L}}$, and all other $N - 1$ eigenvalues have positive real parts. For convenience, let us rearrange
$$J_{\mathcal{L}}=\mbox{block diag}\left(0,J_{N-1}\right),$$
where $J_{N-1}\in \mathds{C}^{(N-1)\times (N-1)}$ consists of the last $(N-1)$ rows and the last $(N-1)$ columns of the matrix  $J_{\mathcal{L}}$. Let $\Phi=\col\left(\Phi_1,\Psi\right)$ and $\Psi=\col\left(\Phi_2\dots,\Phi_N\right)$, where {$\Phi_i\in \mathds{C}^{n\times n}$  for $i=1,\dots,N$}. Then, system \eqref{eS2} can be {rewritten} as
\begin{subequations}
\begin{align}
\dot{\Phi}_1&=0 I_n, \label{eS3-a}\\
\dot{\Psi}&=-\mu_1\left(J_{N-1}\otimes I_n\right)\Psi{.} \label{eS3}
\end{align}\end{subequations}
From equation \eqref{eS3}, and the properties of $J_{N-1}$, we obtain $\lim_{t\rightarrow\infty }\Psi(t)= 0$ exponentially with decay rate $\mu_1 \lambda_{1}$, which implies $$\lim_{t\rightarrow \infty}\Phi(t)=\col\left(\Phi_1(0),0_{(N-1)n\times n}\right)$$ exponentially. Thus, $$\lim_{t\rightarrow \infty} \bar{S}(t)=\left(U\otimes I_n\right)\col\left(\Phi_1(0),0_{(N-1)n\times n}\right)$$ exponentially.
Let $\mathds{1}_N$
be the eigenvector associated to the $0$ simple eigenvalue of $\mathcal{L}$. Then, arrange $U$ so that its first column is $\mathds{1}_N$. Thus, for any positive $\mu_1$ and any initial $S_i\left(0\right)\in \mathds{R}^{n\times n}$,
$\lim_{t\rightarrow\infty }\bar{S}(t)=\left(\mathds{1}_N\otimes \Phi_1(0)\right)$ exponentially, i.e., $\lim_{t\rightarrow\infty }S_i(t)= \Phi_1(0)$, $\forall i$ with decay rate $\mu_1 \lambda_{1}$.
\end{proof}
\begin{rem}\label{rem1chi}
After denoting the first row of  $U^{-1}$ as $u^{T}=\textnormal{\col}\left(u_1,\dots,u_N\right)$, which is a left eigenvector of $\mathcal{L}$ associated with eigenvalue $0$ and belongs to $\mathds{R}^n$. Then, the following equality holds
\begin{align*}
\Phi(0)&=\left(U^{-1}\otimes I_n\right)\bar{S}(0)\\
&=\textnormal{\col}\left(\Phi_1(0),\Phi_2(0),\dots,\Phi_N(0)\right).
\end{align*}
Thus, $ \Phi_1(0)=\sum_{i=1}^{N}u_iS_i(0)$. Denote $S^{*}=\Phi_1(0)$, which can be treated as the system dynamics of the {\myr autonomous system determined by the initial conditions of each agent and communication network}.
\end{rem}
{Next, we show that dynamics \eqref{sweq1ii} achieve consensus to the state of the {\myr autonomous system} constructed by all the agents through communication network.}

\begin{lem}\label{lema2i}Consider the {dynamics \eqref{sweq1ii}} with an arbitrary $\eta_i(0)$. Under Assumption \ref{ass2}, for sufficiently large $\mu_1$ and $\mu_2$, the signals $\eta_i(t)$ achieve consensus exponentially, for $i=1,\dots,N$.
\end{lem}
\begin{proof}
For notational conciseness, let $\eta=\col\left(\eta_1,\dots,\eta_N\right)$ and $\hat{S}_d=\mbox{block diag} \left(S_{1},\dots,S_{N}\right)$, Then, we can put \eqref{sweq1ii} into the following compact form
 \begin{align}\dot{\eta}=&\big[\hat{S}_d- \mu_2\left( \mathcal{L}\otimes I_n\right)\big]\eta.\label{eS1b}\end{align}
 Perform the following transformation
\begin{align}\label{transeta}
\hat{\eta}=P(t)\eta,
\end{align}
where $P(t)=e^{Qt}$ and $Q=\mu_2\left( \mathcal{L}\otimes I_n\right)-\left(I_N\otimes S^{*}\right)$.
The time derivative of $\hat{\eta}$ along the trajectory \eqref{eS1b} is
\begin{align}\label{sweq103}
    \dot{\hat{\eta}}
   =& P(t)\big[\hat{S}_d(t)-\left(I_N\otimes S^{*}\right)\big] P^{-1}(t)\hat{\eta}\nonumber\\
   =&e^{Qt}\big[\hat{S}_d(t)-\left(I_N\otimes S^{*}\right)\big] e^{-Qt}\hat{\eta}\nonumber\\
   =& F(t)\hat{\eta}.
\end{align}
We know from Lemma \ref{lema2} that $\lim_{t\rightarrow\infty}S_i(t)= S^{*}$ exponentially with decay rate $\mu_1\lambda_1$. Note that $\|e^{Q t}\|$ and $\|e^{-Q t}\|$ are upper bounded by $e^{(\mu_2\|\mathcal{L}\|+\|S^{*}\|)t}$.
Therefore, we have $\lim_{t\rightarrow\infty}F(t)= 0$ exponentially for $$\mu_1\geq 2(\mu_2\|\mathcal{L}\|+\|S^{*}\|)/\lambda_1.$$ Then, by Lemma \ref{lema1}, for any initial states $\hat{\eta}\left(0\right)\in \mathds{R}^{Nn}$, $\hat{\eta}\left(t\right)$ converges to a bounded vector $\hat{\eta}^{*}=\col\left(\hat{\eta}^{*}_1,\dots,\hat{\eta}^{*}_N\right)$, 
 $\hat{\eta}^{*}_i\in \mathds{R}^{n}$. Since graph $\mathcal{G}$ contains a spanning tree, for any positive $\mu_2$ and any initial $\hat{\eta}\left(0\right)$, we have from Lemma $\ref{WRWB2005}$ that
\begin{align}\label{transeta0}
\lim \limits_{t\rightarrow\infty} e^{- \mu_2\left( \mathcal{L}\otimes I_n\right)t}\hat{\eta}\left(t\right) =&\lim \limits_{t\rightarrow\infty} e^{- \mu_2\left( \mathcal{L}\otimes I_n\right)t} \lim \limits_{t\rightarrow\infty}\hat{\eta}\left(t\right)\nonumber\\
  =&\mathds{1}_N\otimes \chi^*
\end{align}
where $\chi^*=\sum_{i=1}^{N}u_i\hat{\eta}^{*}_{i}$ and $u_i$ is defined in Remark \ref{rem1chi}.
 Let \begin{align}\eta_0(t)=\mathds{1}_N\otimes \big(e^{S^*t}\chi^*\big).\label{consesusprool}\end{align}
According to \eqref{transeta}, we have, $$\eta(t)=e^{-Qt}\hat{\eta}(t) =e^{\left(I_N\otimes S^{*}\right)t} e^{-\mu_2\left(  \mathcal{L}\otimes I_n\right)t}\hat{\eta}(t).$$
Since $\left\|e^{\left(I_N\otimes S^{*}\right)t}\right\|\leq e^{\|S\|t}$,
\begin{align*}
\left\|\eta(t)-\eta_0(t)\right\|&=\left\|e^{\left(I_N\otimes S^{*}\right)t}\left[ e^{-\mu_2\left(  \mathcal{L}\otimes I_n\right)t}\hat{\eta}(t)-\mathds{1}_N\otimes \chi^*\right]\right\|\\
&\leq e^{\|S^{*}\|t}\left\|e^{-\mu_2\left(  \mathcal{L}\otimes I_n\right)t}\hat{\eta}(t)-\mathds{1}_N\otimes \chi^*\right\|.
\end{align*}
Considering \eqref{transeta0}, the exponentially decay rate is $\mu_2 \lambda_{1}$. Then, we have
\begin{align*}
\left\|\eta(t)-\eta_0(t)\right\|&\leq e^{\|S^{*}\|t}e^{-\mu_2 \lambda_{1} t}\nonumber\\
&=e^{-(\mu_2 \lambda_{1}-\|S^{*}\|) t}.
\end{align*}
Hence, for $i,j \in \mathcal{N}$ and $\mu_2>\frac{\|S^{*}\|}{\lambda_{1}}$,
\begin{equation}
\lim\limits_{t\rightarrow\infty}\left(\eta_i\left(t\right)-\eta_j\left(t\right)\right)=0 \label{conzero}
\end{equation}
exponentially. This further implies that $$\lim\limits_{t\rightarrow\infty}(\eta_i(t)- e^{S^*t}\chi^*)=0$$ exponentially for all $i$.
\end{proof}
\begin{rem}
{\myr  Note that the convergence analysis of Lemma \ref{lema2i} does not require the consensus state to be bounded, whereas the convergence analysis in some recent works such as \cite{cai2016leader} relies on the condition that the state of the leader is bounded.} The idea of constructing {\myr an autonomous system} in a distributed way was proposed in \cite{baldi2019leaderless} for agents in the form of heterogeneous oscillators over undirected graphs. More specifically, in \cite{baldi2019leaderless} the matrix $S_i$ {takes} the following form
\begin{align}\label{transeta00}
S_i=\left[
        \begin{array}{cc}
          0 & 1 \\
          -\beta_i& 0 \\
        \end{array}
      \right]
\end{align}
together with the following distributed dynamics
 $$\dot{\beta}_i=\sum\nolimits_{j\in \mathcal{N}_i}{a_{ij}(\beta_j-\beta_i)},$$
 where $\beta_i\in \mathds{R}$. Lemma \ref{lema2i} extends this result to directed graphs and more general $S_i$. In the following, an internal model design is discussed to handle the unknown disturbances.
\end{rem}

\subsection{{Design of a dynamic compensator}}

A so-called internal model approach can be adopted to reject the disturbances $d_i(t)$. For compactness, let $\sigma_{is}=\col(\sigma_{is,1},\dots,\sigma_{is,n_{is}})$ and $\sigma_{i}=\col\left(\sigma_{i1},\dots,\sigma_{in}\right)$, $i=1,\dots,N$ and $s=1,\dots,n$. According to \cite{serrani2001semi,chen2009attitude,isidori2003robust,isidori2014robust,chen2014attitude,lu2019adaptive,nikiforov1998adaptive}, we know that for each $i=1,\dots,N$ and $s=1,\dots,n$, there exist positive integers $r_{is}$ and real numbers $c_{is,1},\dots,c_{is,r_{is}}$ which may depend on $\sigma_{is}$, such that
\begin{equation*}
d_{is}^{(r_{is})}= c_{is,1}d_{is} + c_{is,2}\dot{d}_{is} +   \dots + c_{is,r_{is}}d^{(r_{is}-1)}_{is}.
\end{equation*}
Let $T_{is}^{\sigma_{is}}$ be {a nonsingular} matrix of dimension $r_{is}$, and
$$\vartheta_{is}=\col\big(d_{is},\dot{d}_{is},d_{is}^{(2)},\dots,d_{is}^{(r_{is}-1)}\big).$$
Then, we have
\begin{equation*}
\dot{\vartheta}_{is}=\Phi_{is}^{\sigma_{is}}\vartheta_{is},~~d_{is}=\Psi_{is}\vartheta_{is},
\end{equation*}
where
\begin{align*}
\Phi_{is}^{\sigma_{is}}=&\left[
         \begin{array}{c|c}
           0 &  I_{r_{is}-1} \\
           \hline
           c_{is,1} & c_{is,2},\dots,c_{is,r_{is}}
         \end{array}
       \right], \;
\Psi_{is}=\row(1,0_{r_{is}-1}).
\end{align*}
Let $M_{is}\in \mathds{R}^{r_{is}\times r_{is}}$ be Hurwitz, $N_{is}\in \mathds{R}^{r_{is}}$, and $(M_{is},N_{is})$ be controllable. Then, there exists a nonsingular matrix $T_{is}^{\sigma_{is}}$ satisfying the Sylvester equation
\begin{equation}\label{Tmatrix}
T_{is}^{\sigma_{is}}\Phi_{is}^{\sigma_{is}}-M_{is}T_{is}^{\sigma_{is}}=N_{is}\Psi_{is}.
\end{equation}
Let $\theta_{is}(t) =-T_{is}^{\sigma_{is}} \vartheta_{is}(t)$, $\theta_{i}=\textnormal{\col}(\theta_{i1},\dots,\theta_{in})$,
  $\Psi_i = \mbox{block diag}(\Psi_{i1},\dots,\Psi_{in})$,
  $M_i = \mbox{block diag}(M_{i1},\dots,M_{in})$, $T_i^{\sigma_{i}}=\mbox{block diag}(T_{i1}^{\sigma_{i1}},\dots,T_{in}^{\sigma_{in}})$, and $N_i =\mbox{block diag}(N_{i1},\dots,N_{in})$.
Then we have $$d_i=-\Psi_{i}\left(T_{i}^{\sigma_{i}}\right)^{-1}\theta_{i}.$$
The dynamic compensator is designed as
\begin{equation}\label{dstureq4}
\dot{\xi}_i=M_i\xi_i+N_i\tau_i,
\end{equation}
where $\xi_i\in \mathds{R}^{n_i}$ with $n_i=\sum_{s=1}^{n} r_{is}$. The next section concerns the design of the distributed control $\tau_i$.
\section{Main Results} \label{section4}
To propose a distributed control law for the EL agents, we assume that $\dot{\eta}_0=\mathds{1}_N\otimes (S^*e^{S^*t}\chi^{*})$ in \eqref{consesusprool} is bounded for all $t\geq0$, which implies that $\dot{\eta}_i$ is bounded for all $t\geq0$, for $i=1,\dots,N$. Let
\begin{subequations}\label{claw}
\begin{align}
\dot{q}_{ri}&=S_i\eta_i-\alpha\left(q_i-\eta_i\right),\label{claw1}\\
 s_i&=\dot{q}_i-\dot{q}_{ri},\label{claw2}
 \end{align}
\end{subequations}
where $\alpha>0$ and $\eta_i$ and $S_i$ {are} generated by \eqref{sweq1}. Then,
\begin{subequations}\label{dclaw}
\begin{align}
\ddot{q}_{ri}&=\dot{S}_i{\eta}_i+S_i\dot{\eta}_i-\alpha\left(\dot{q}_i-\dot{\eta}_i\right),\label{dclaw1}\\
 \dot{s}_i&=\ddot{q}_i-\ddot{q}_{ri}.\label{dclaw2}
 \end{align}
\end{subequations}
By Property \ref{prop1}, there exists a known matrix $Y_i\left(q_i,\dot{q}_i, {\myr \dot{q}_{ri}, \ddot{q}_{ri}}\right)$ and an unknown constant vector $\Theta_i$ such that
\begin{align}\label{Ybounded}
 Y_i\left(q_i,\dot{q}_i, {\myr \dot{q}_{ri}, \ddot{q}_{ri}}\right) \Theta_i=&\Mathcal{M}_i\left(q_i\right)\ddot{q}_{ri}+G_i\left(q_i\right)\nonumber\\
 &+\Mathcal{C}_i\left(q_i,\dot{q}_i\right)\dot{q}_{ri}.
\end{align}
Next, substituting $Y_i\left(q_i,\dot{q}_i, {\myr \dot{q}_{ri}, \ddot{q}_{ri}}\right)\Theta_i$ into system \eqref{MARINEVESSEL1} gives
\begin{align}\label{undistur2}
  \Mathcal{M}_i\left(q_i\right)\left(\ddot{q}_i-\ddot{q}_{ri}\right)+\Mathcal{C}_i\left(q_i,\dot{q}_i\right)\left(\dot{q}_i-\dot{q}_{ri}\right) &\nonumber\\
+Y_i\left(q_i,\dot{q}_i,{\myr \dot{q}_{ri}, \ddot{q}_{ri}}\right)\Theta_i&=\tau_i+d_i.
\end{align}
Then, from \eqref{claw2} and \eqref{undistur2}, we have
\begin{align}
\Mathcal{M}_i\left(q_i\right)\dot{s}_i=&\tau_i-\Mathcal{C}_i\left(q_i,\dot{q}_i\right)s_i\nonumber\\
&-Y_i\left(q_i,\dot{q}_i,{\myr \dot{q}_{ri}, \ddot{q}_{ri}}\right)\Theta_i+d_i.\label{undistur1i}
\end{align}
Consider the augmented system composed of (\ref{dstureq4}) and (\ref{undistur1i}), and the following coordinate transformation
\begin{subequations}
\begin{align}
\bar{\xi}_i &=\xi_i-\theta_i \label{20201216-contr-xi}\\
\tilde{\tau}_i&=\tau_i-A_i\xi_i \label{20201216-contr-tau}\\
d_i&=-B_i\theta_{i}
\end{align}
\end{subequations}
where $A_i=\Psi_i\left(T_{i}^{0}\right)^{-1}$ and $B_i=\Psi_i\left(T_{i}^{\sigma_i}\right)^{-1}$ with $T_i^{0}$ being a nonsingular matrix, $\Psi_i$ and $T_{i}^{\sigma_i}$ given in \eqref{Tmatrix}. We have
\begin{align*}
\dot{\bar{\xi}}_i =&\left[M_i+N_iA_i\right]\bar{\xi}_i+N_i\check{u}+N_iE^{\sigma_i}_{i}\theta_{i},\\
\Mathcal{M}_i\left(q_i\right)\dot{s}_i=&\tilde{\tau}_i-\Mathcal{C}_i\left(q_i,\dot{q}_i\right)s_i+A_i\bar{\xi}_i\nonumber\\
 &-Y_i\left(q_i,\dot{q}_i,{\myr \dot{q}_{ri}, \ddot{q}_{ri}}\right)\Theta_i+E^{\sigma_i}_{i}\theta_{i},
 \end{align*}
with $E^{\sigma_i}_{i}= A_i-B_i$. Then, a further transformation $$\tilde{\xi}_i=\bar{\xi}_i-N_i\mathcal{M}_i\left(q_i\right)s_i,$$ gives
 \begin{align*}
\dot{\tilde{\xi}}_i
=&M_i\tilde{\xi}_i+P_i(q_i,\dot{q}_i,s_i)\Theta_i,\\
\Mathcal{M}_i\left(q_i\right)\dot{s}_i
=&\tilde{\tau}_i-\Mathcal{C}_i\left(q_i,\dot{q}_i\right)s_i+A_i\tilde{\xi}_i+Q_i(q_i,\dot{q}_i,s_i)\Theta_i\nonumber\\
&+E^{\sigma_i}_{i}\xi_i-E^{\sigma_i}_{i}\left[\tilde{\xi}_i+N_i\mathcal{M}_i\left(q_i\right)s_i\right],
 \end{align*}
where $\xi_i\in \mathds{R}^{n_i}$, $s_i\in \mathds{R}^{n}$, and
\begin{align}
  P_i(q_i,\dot{q}_i,s_i)\Theta_i =& M_iN_i\mathcal{M}_i\left(q_i\right)s_i+N_i\mathcal{C}_i\left(q_i,\dot{q}_i\right)s_i\nonumber\\
  &-N_i\dot{\mathcal{M}}_i\left(q_i\right)s_i +N_iY_i\left(q_i,\dot{q}_i,{\myr \dot{q}_{ri}, \ddot{q}_{ri}}\right)\Theta_i,\nonumber\\
  Q_i(q_i,\dot{q}_i,s_i)\Theta_i=& A_iN_i\mathcal{M}_i\left(q_i\right)s_i \nonumber\\
  &-Y_i\left(q_i,\dot{q}_i,{\myr \dot{q}_{ri}, \ddot{q}_{ri}}\right)\Theta_i,\label{PQbounded}
\end{align}
{with} $P_i(q_i,\dot{q}_i,s_i)$ and $Q_i(q_i,\dot{q}_i,s_i)$ {being} known regression matrices.
Let $\zeta_i\in R^{n_i\times p}$ be produced by an auxiliary system
\begin{equation}\label{dstureq8}
  \dot{\zeta}_i=M_i\zeta_i+P_i(q_i,\dot{q}_i,s_i).
\end{equation}
Let $\hat{\xi}_i=\tilde{\xi}_i-\zeta_i\Theta_i$. A straightforward computation shows
\begin{subequations}\label{dstureq9}
 \begin{align}
   \dot{\hat{\xi}}_i=&M_i\tilde{\xi}_i+P_i(q_i,\dot{q}_i,s_i)\Theta_i\nonumber\\
                     &-\left[M_i\zeta_i+P_i(q_i,\dot{q}_i,s_i)\right]\Theta_i\nonumber\\
                    =&M_i\hat{\xi}_i, \label{dstureq9i}\\
 \Mathcal{M}_i\left(q_i\right)\dot{s}_i
=&\tilde{\tau}_i-\Mathcal{C}_i\left(q_i,\dot{q}_i\right)s_i+B_i\hat{\xi}_i\nonumber\\
&+\left[A_i\zeta_i+Q_i(q_i,\dot{q}_i,s_i)\right]\Theta_i\nonumber\\
&+E^{\sigma_i}_{i}\left[\xi_i-N_i\mathcal{M}_i\left(q_i\right)s_i\right]-E^{\sigma_i}_{i}\zeta_i\Theta_i. \label{dstureq9ii}
 \end{align}
\end{subequations}
Since {\myr $M_i$ is Hurwitz in \eqref{dstureq9i}}, we only need to concentrate on the second equation of (\ref{dstureq9}). To handle the uncertain term in \eqref{dstureq9ii} (i.e. the last two lines of \eqref{dstureq9ii}) with adaptive control technique, we note that the uncertainty in the matrix $E^{\sigma_i}_{i}$ can be linearly parameterized for some integer $l\geq 1$ as follows
\begin{align*}
  E^{\sigma_i}_{i} =& \sum\nolimits_{j=1}^{l}E_{ij}\varrho_{ij}\\
  =&E_i\left[\varrho_{i}\otimes I_{n_i}\right],
\end{align*}
where $E_i=\row\left(E_{i1},\dots,E_{il}\right)$, $\varrho_{i}=\col\left(\varrho_{i1},\dots,\varrho_{il}\right)$, $E_{ij}\in \mathds{R}^{n\times n_i}$ is a constant matrix and $\varrho_{ij}\in \mathds{R}$ is a smooth function of $\sigma_i$. As a result
\begin{eqnarray*}
E^{\sigma_i}_{i}\zeta_i\Theta_i=\left[E_i\circ \zeta_i\right]\left[\varrho_{i}\otimes\Theta_i \right],
\end{eqnarray*}
{where } $E_i\circ \zeta_i=\row\left(E_{i1}\zeta_i,\dots, E_{il}\zeta_i\right)$. Besides,
\begin{align*}
E^{\sigma_i}_{i}\xi_i=&\left[E_i\circ\xi_i\right]\varrho_{i},\\
E^{\sigma_i}_{i}N_iM_i\left(q_i\right)s_i=&E_i\left[\varrho_{i}\otimes I_{n_i}\right] N_iL_i\left(q_i,s_i\right)\Theta_i \nonumber\\
=&E_i\circ \left[N_iL_i\left(q_i,s_i\right)\right]\left[\varrho_{i}\otimes\Theta_i \right],
\end{align*}
where $ L_i\left(q_i,s_i\right)\Theta_i=\Mathcal{M}_i\left(q_i\right)s_i$, {and} $L_i\left(q_i,s_i\right)$ is a known regression matrix. Now, the system (\ref{dstureq9}) can be further written in the following linearly parameterized form
\begin{subequations}\label{dstureq10}
 \begin{align}
   \dot{\hat{\xi}}_i=&M_i\hat{\xi}_i, \label{dstureq10i}\\
 \Mathcal{M}_i\left(q_i\right)\dot{s}_i
=&\tilde{\tau}_i-\Mathcal{C}_i\left(q_i,\dot{q}_i\right)s_i\nonumber\\
&+B_i\hat{\xi}_i+\rho_i(q_i,\dot{q}_i,s_i,\zeta_i)\omega_i,\label{dstureq10ii}
 \end{align}
\end{subequations}
where $\omega_i  =\col\left(\Theta_i,\varrho_{i}\otimes\Theta_i,\varrho_{i}\right)$  is a constant vector consisting of the uncertain parameters of (\ref{MARINEVESSEL1}) and \eqref{exdistur}, and $\rho_i(q_i,\dot{q}_i,s_i,\zeta_i)$ is a known regression matrix with
\begin{align}
      \rho_i(q_i,\dot{q}_i,s_i,\zeta_i)=&\col( \rho_{i1}(q_i,\dot{q}_i,s_i,{\myr \zeta_i}),\rho_{i2}(q_i,s_i,\zeta_i), \rho_{i3}(\xi_i)  )\nonumber\\
      =&\left[
                                           \begin{array}{c}
                                             A_i\zeta_i+Q_i(q_i,\dot{q}_i,s_i) \\
                                            E_i\circ \left[\zeta_i+N_iL_i\left(q_i,s_i\right)\right]\\
                                            E_i\circ\xi_i\\
                                           \end{array}
                                         \right].\label{rhobounded}
      \end{align}
The last step for solving the regulation problem of system (\ref{dstureq10}) is to introduce the
control law as follows,
\begin{subequations}\label{dstureq11}
 \begin{align}
  \tilde{\tau}_i&=-K_{i}s_i-\rho_i(q_i,\dot{q}_i,s_i,\zeta_i)\hat{\myr \omega}_i,\\
   \dot{\hat{\omega}}_i&=\Lambda_{i}^{-1}\rho_{i}^T(q_i,\dot{q}_i,s_i,\zeta_i)s_i, \label{dstureq11ii}
 \end{align}
\end{subequations}
where $s_i$ is calculated from \eqref{claw2}, $\zeta_i$ is generated by \eqref{dstureq8}, the vector $\hat{\omega}_i$ is used to estimate $\omega_i$, $K_i$ is a positive definite matrix, and $\Lambda_i$ a positive definite diagonal matrix representing the estimator update rate. Now we are in a position to present our main result.
\begin{thm}
Consider system (\ref{MARINEVESSEL1}) over a communication graph satisfying Assumption \ref{ass2}. Problem \ref{ldlesp} is solvable by the control law consisting of \eqref{sweq1}, \eqref{dstureq4}, \eqref{dstureq8}, and \eqref{dstureq11}. 
\end{thm}
\begin{proof}Substituting (\ref{dstureq11}) into (\ref{dstureq10}) gives
 \begin{subequations}\label{MARINEVESSEL14}
\begin{align}
\dot{\hat{\xi}}_i=&M_i\hat{\xi}_i, \\
  \Mathcal{M}_i\left(q_i\right)\dot{s}_i=&-\Mathcal{C}_i\left(q_i,\dot{q}_i\right)s_i-K_is_i\nonumber\\
  &-\rho_i(q_i,\dot{q}_i,s_i,\zeta_i)\tilde{\omega}_i+B_i\hat{\xi}_i, \label{MARINEVESSEL14i}\\
  \dot{\tilde{\omega}}_i=&\Lambda_{i}^{-1}\rho_{i}^T(q_i,\dot{q}_i,s_i,\zeta_i)s_i,\label{MARINEVESSEL14ii}
\end{align}
\end{subequations}
where $\tilde{\omega}_i=\hat{\omega}_i-\omega_i$. Let $\mathcal{Q}_i$ be the symmetric positive definite matrix satisfying $$\mathcal{Q}_iM_i+M_i^{T}\mathcal{Q}_i=-I$$ and pick a real number $\epsilon\geq\frac{\|B_i\|^2}{\lambda_{min}(K_i)}$, where $\|B_i\|=\max\limits_{\|x\|=1}\|B_ix\|$. Pick the following Lyapunov function candidate
\begin{equation*}
V_i=\epsilon\hat{\xi}_{i}^T\mathcal{Q}\hat{\xi}_{i}+\frac{1}{2}\left[s_{i}^T\mathcal{M}_i\left(q_i\right)s_i+\tilde{\omega}_{i}^T\Lambda_{i}\tilde{\omega}_{i}\right].
\end{equation*}
The time derivative of $V$ along the trajectory \eqref{MARINEVESSEL14} is
\begin{align*}
  \dot{V}_i
  =&-s_{i}^TK_is_i+\frac{1}{2}s_{i}^T\big[\dot{\mathcal{M}}_i\left(q_i\right)-2\Mathcal{C}_i\left(q_i,\dot{q}_i\right)\big]s_i+s_{i}^TB_i\hat{\xi}_i\nonumber\\
  &-s_{i}^T\rho_{i}(q_i,\dot{q}_i,s_i,\zeta_i)\tilde{\omega}_i+\tilde{\omega}_{i}^T\rho_{i}^T(q_i,\dot{q}_i,s_i,\zeta_i)s_i-\epsilon\|\hat{\xi}_{i}\|^2.\nonumber
  \end{align*}
  Since $\dot{{\mathcal{M}}}_i\left(q_i\right)-2\Mathcal{C}_i\left(q_i,\dot{q}_i\right)$ is skew symmetric, we have
\begin{align}\label{MARINEVESSEL16}
  \dot{V}_i=&-s_{i}^TK_is_i+s_{i}^TB_i\hat{\xi}_i-\epsilon\|\hat{\xi}_{i}\|^2\nonumber\\
  \leq&-s_{i}^TK_is_i+\frac{1}{2\epsilon}\|s_{i}^TB_i\|^2+\frac{\epsilon}{2}\|\hat{\xi}_i\|^2-\epsilon\|\hat{\xi}_{i}\|^2\nonumber\\
  \leq&-\frac{\epsilon}{2}\|\hat{\xi}_i\|^2-\frac{1}{2}s_{i}^TK_is_i\nonumber\\
  =&-a\big(\hat{\xi}_i,s_{i}\big).
  \end{align}
Thus, $s_{i}$, $\hat{\xi}_{i}$ and $\tilde{\omega}_i$ are bounded. From(\ref{sweq1}) and (\ref{claw}), we have
   \begin{align*}
   \dot{q}_{i}-&\dot{\eta}_i+\alpha\left(q_{i}-\eta_i\right)
  =s_i- \mu_2\sum\nolimits_{j\in \mathcal{N}_i} {a_{ij}(\eta_j-\eta_i)},
  \end{align*}
which can be further rewritten as
   \begin{equation}\label{errorsystem}
  \dot{e}_i+\alpha e_i=s_i- \mu_2\sum\nolimits_{j\in \mathcal{N}_i} {a_{ij}(\eta_j-\eta_i)}.
  \end{equation}
This can be viewed as a stable first-order differential equation in $e_i$ with $s_i-\mu_2\sum\nolimits_{j\in \mathcal{N}_i} {a_{ij}(\eta_j-\eta_i)}$ as the input. Since this input is bounded for all $t\geq 0$, we conclude that both $e_i=q_i-\eta_i$ and $\dot{e}_i=\dot{q}_i-\dot{\eta}_i$ are bounded for all $t\geq 0$, which further implies $\dot{q}_i$ is bounded for all $t\geq 0$ because of $\dot{\eta}_i$ is bounded for all $t\geq 0$.

By Property \ref{propbound}, we obtain that $\Mathcal{M}_i(q_i)$, $\Mathcal{C}_i\left(q_i,\dot{q}_i\right)$ and $G_i(q_i)$ are all bounded for all $t\geq 0$.
It is noted that $$\lim_{t\rightarrow\infty}\dot{S}_i(t)\eta_i(t)=0\;\textnormal{and}\; \lim_{t\rightarrow\infty}\big[S_i(t)\dot{\eta}_i(t)- (S^*)^2e^{S^*t}\chi^{*}\big]=0$$ from Lemma~\ref{lema2} and Lemma~\ref{lema2i},
where $\chi^{*}$ and $S^*$ are defined in \eqref{transeta0} and Remark \ref{rem1chi}, respectively. Hence, {\myr $\dot{q}_{ri}$ and $\ddot{q}_{ri}$ are bounded from \eqref{claw} for all $t\geq0$.}
By equation \eqref{MARINEVESSEL14i}, we have $Y_i\left(q_i,\dot{q}_i,{\myr \dot{q}_{ri}, \ddot{q}_{ri}}\right)$ is bounded. Noted that, $P_i(q_i,\dot{q}_i,s_i)$ and $Q_i(q_i,\dot{q}_i,s_i)$ are bounded for all $t\geq0$ from \eqref{PQbounded}. Thus, $\zeta_i$ is also bounded for all $t\geq0$ from a stable differential equation \eqref{dstureq8} with a bounded input $P_i(q_i,\dot{q}_i,s_i)$. As a result, $\xi_i$ is bounded for all $t\geq0$ from \eqref{20201216-contr-xi} and the fact that $\theta_i$ is bounded for all $t\geq0$.
Then, $\rho(q_i,\dot{q}_i,s_i,\zeta_i)$ is bounded for all $t\geq0$ from \eqref{rhobounded}. Hence, we have $\dot{s}(t)$ is bounded for all $t\geq0$ from \eqref{MARINEVESSEL14i}.

By integrating both sides of (\ref{MARINEVESSEL16}), we can show that
$$\int_{0}^{t}a\big(\hat{\xi}_i(\tau),s_{i}(\tau)\big)d\tau \leq V(0)- V(t) \leq V(0). $$
  Thus $\lim\limits_{t\rightarrow\infty}\int_{0}^{t}a\big(\hat{\xi}_i(\tau),s_{i}(\tau)\big) d\tau$ exists and is finite. Therefore, $$\dot{a}\big(\hat{\xi}_i(t),s_{i}(t)\big)=\frac{\partial a}{\partial \hat{\xi}_i}\dot{\hat{\xi}}_i + \frac{\partial a}{\partial s_i}\dot{s}_i$$ is bounded for all $t\geq 0$, and hence ${a}\big(\hat{\xi}_i(t),s_{i}(t)\big)$ is uniformly continuous in $t$.  Applying Barbalat's lemma, we have $\lim_{t\rightarrow \infty}{a}\big(\hat{\xi}_i(t),s_{i}(t)\big) = 0$, thus $\lim_{t\rightarrow \infty}s_i(t)= 0$.

Since the input in \eqref{errorsystem} is bounded for all $t\geq 0$ and tends to zero as $t\rightarrow \infty$, we conclude that both $e_i=q_i-\eta_i$ and $\dot{e}_i=\dot{q}_i-\dot{\eta}_i$ are bounded for all $t\geq 0$ and will decay to zero. Together with \eqref{conzero}, the proof is {completed}.
\end{proof}

\begin{rem}
For the single Euler-Lagrange system as in \cite{lu2019adaptive}, the tracking signal is bounded. In our multiple Euler-Lagrange setting we only require that the derivative of the final consensus state is bounded without
imposing bounds on the cooperatively agreed trajectory. 
\end{rem}

\section{Numerical Example}\label{section5}
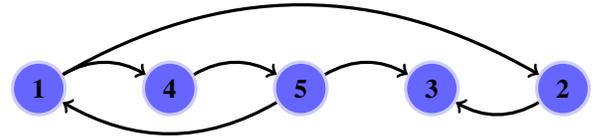
\begin{figure}[ht]
\begin{center}
\begin{tikzpicture}[transform shape]
    \centering%

    \node (3) [circle,draw=blue!20, fill=blue!60, very thick, minimum size=7mm] {\textbf{3}};
     \node (5) [circle,  left =of 3, draw=blue!20, fill=blue!60, very thick, minimum size=7mm] {\textbf{5}};
     \node (4) [circle, left=of 5, draw=blue!20, fill=blue!60, very thick, minimum size=7mm] {\textbf{4}};
   \node (1) [circle, left=of 4, draw=blue!20, fill=blue!60, very thick, minimum size=7mm] {\textbf{1}};
   \node (2) [circle, right=of 3, draw=blue!20, fill=blue!60, very thick, minimum size=7mm] {\textbf{2}};
    \draw[ very  thick,->,bend  left] (1) edge (2);
    \draw[ very  thick,->, bend left] (1) edge (4);
    \draw[ very  thick,->, bend left] (5) edge (3);
    \draw[ very  thick,->, bend  left] (2) edge (3);
    \draw[ very  thick,->, bend left] (4) edge (5);
    \draw[ very  thick,->,bend left] (5) edge (1);
\end{tikzpicture}
\end{center}
\caption{ Communication {graph} $\bar{\mathcal{G}}$}\label{fig0}
\end{figure}

\begin{figure}
  \centering
   \includegraphics[width=0.93\linewidth]{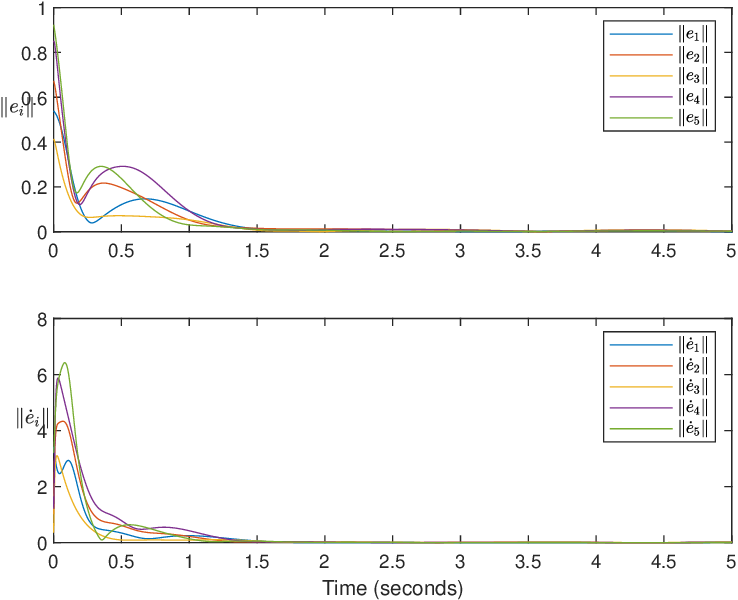}\\
   \caption{{\hwz Trajectories} of $\|e_i\|$ and $\|\dot{e}_i\|$, for $i=1,\dots,5$.}\label{fig2}
\end{figure}
%

\begin{figure}
  \centering
   \includegraphics[width=0.93\linewidth]{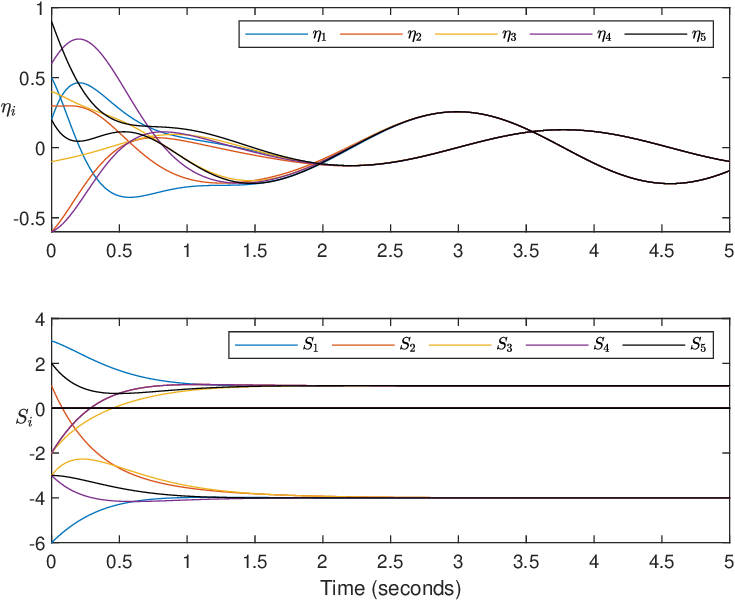}\\
  \caption{{\hwz Trajectories} of $\eta_i$ and $S_i$, for $i=1,\cdots,5$.}\label{fig3}
\end{figure}

Consider a group of 5 EL agents with the communication network described in Figure \ref{fig0}. Let each EL agent represent a two-link robotic arm, whose dynamics is described by \eqref{MARINEVESSEL1}, {with} generalized coordinates $q_i=\col\left(\theta_{i1},\theta_{i2}\right)$,
\begin{subequations}
\begin{align}
  \Mathcal{M}_i\left(q_i\right) &= \left[
                            \begin{array}{cc}
                              a_{i1} +a_{i2}+2a_{i3}\cos\theta_{i2} &  a_{i2}+a_{i3}\cos\theta_{i2} \\
                               a_{i2}+ a_{i3}\cos\theta_{i2}&  a_{i2} \\
                            \end{array}
                          \right],\nonumber \\
  \Mathcal{C}_i\left(q_i,\dot{q}_i\right) &= \left[
                                      \begin{array}{cc}
                                        -a_{i3}\left(\sin\theta_{i2}\right)\dot{\theta}_{i2}& -a_{i3}\sin\theta_{i2}\big(\dot{\theta}_{i1}+\dot{\theta}_{i2}\big) \\
                                         a_{i3}\sin\theta_{i2}\dot{\theta}_{i1} & 0  \\
                                      \end{array}
                                    \right],
  \nonumber\\
 G_i\left(q_i\right) &=\left[
                          \begin{array}{c}
                            a_{i4}g\cos\theta_{i1}+a_{i5}g\cos\left(\theta_{i1}+\theta_{i2}\right) \\
                            a_{i5}g\cos\left(\theta_{i1}+\theta_{i2}\right) \\
                          \end{array}
                        \right],
 \nonumber
\end{align}
\end{subequations}
and $\Theta_i=\col\left(a_{i1},a_{i2},a_{i3},a_{i4},a_{i5}\right)$. This dynamics is adopted from Example 3.2-2 in \cite{lewis2004Book-robot} with some simplified modification of notations. The physical interpretation of each parameter can be found in \cite{lewis2004Book-robot}.
We consider the disturbance $$d_{ik}=\psi_{ik}\sin\left(\sigma_{ik}t+\phi_{ik}\right),~~~k=1,2.$$
According to the internal model approach, we can select 
$$\Phi_{ik}=\left[
\begin{array}{cc}
 0 & 1 \\
  -\sigma_{ik}^2 & 0 \\
  \end{array}
  \right],~~\Psi_{ik}=\left[
    \begin{array}{cc}
    1 &
     0 \\
     \end{array}
     \right].$$ Choosing $$M_{ik}=\left[
   \begin{array}{cc}
    0 & 1 \\
    -3 & -2 \\
    \end{array}
    \right],~~N_{ik}=\left[
    \begin{array}{c}
    0 \\
     1 \\
     \end{array}
     \right],$$ gives
     \begin{align*}
      T_{ik}^{\sigma_{ik}} =& \left[
       \begin{array}{cc}
       3-\sigma_{ik}^2 & -2 \\
       2\sigma_{ik}^2 & 3-{\sigma_{ik}^2} \\
       \end{array}
       \right]\frac{1}{\left(3-\sigma_{ik}^2\right)^2+4\sigma_{ik}^2},\\
       \Psi_{ik}\left(T_{ik}^{\sigma_{ik}}\right)^{-1}=&\left[
       \begin{array}{cc}
       3-\sigma_{ik}^2 & 2 \\
       \end{array}
       \right],\\
T_{ik}^{0} =& \left[
       \begin{array}{cc}
       3 & -2 \\
      0 & 3 \\
       \end{array}
       \right]\frac{1}{9}.
\end{align*}
Let $\sigma_i=\col\left(\sigma_{i1}, \sigma_{i2}\right)$, $\psi_i=\col\left(\psi_{i1}, \psi_{i2}\right)$, $\phi_i=\col\left(\phi_{i1}, \phi_{i2}\right)$, $M_i=\mbox{block diag}\left(M_{i1},M_{i2}\right)$, $N_i=\mbox{block diag}\left(N_{i1},N_{i2}\right)$, $T_i=\mbox{block diag}\left(T_{i1},T_{i2}\right)$, and $ \Psi_{i}=\mbox{block diag} \left(\Psi_{i1},\Psi_{i2}\right) $. For the nominal value $\sigma_i=0$, we have
\begin{align*}
E^{\sigma_i}_{i}= & \Psi_i\left(T_{i}^{0}\right)^{-1}-\Psi_i\left(T_{i}^{\sigma_i}\right)^{-1}=\left[
                                                                             \begin{array}{cccc}
                                                                                \sigma_{i1}^2 &0 & 0& 0 \\
                                                                                 0& 0& \sigma_{i2}^2&  0 \\
                                                                             \end{array}
                                                                           \right]\\
=&\left[
     \begin{array}{cccc|cccc}
       1 & 0 & 0 & 0 & 0 & 0 & 0 & 0 \\
       0 & 0 & 0 & 0& 1 & 0 & 0 & 0 \\
     \end{array}
   \right]
\left[\left[
        \begin{array}{c}
          \varrho_{i1} \\
          \varrho_{i2}\\
        \end{array}
      \right]\otimes I_4\right]\\
=&\left[
     \begin{array}{cc}
     E_{i1}&E_{i2}\\
     \end{array}
   \right]\left[\varrho_i\otimes I_4\right]\\
 =&E_i\left[\varrho_i\otimes I_4\right],
 \end{align*}
where $\varrho_i=\col\left(\varrho_{i1}, \varrho_{i2}\right)=\col\left(\sigma_{i1}^2,\sigma_{i2}^2\right)$. {Then, we have $\omega_i  =\col\left(\Theta_i,\varrho_{i}\otimes\Theta_i,\varrho_{i}\right)$.} Next, the terms $\rho_i(q_i,\dot{q}_i,s_i,\zeta_i)$ and $P_i(q_i,\dot{q}_i,s_i)$ can be obtained from the
following equations
\begin{align*}
          \rho_i(q_i,\dot{q}_i,s_i,\zeta_i) =&\left[
                                                \begin{array}{c}
                                                  A_i\zeta_i+Q_i(q_i,\dot{q}_i,s_i) \\
                                                  E_i\circ \left[\zeta_i+NL_i\left(q_i,s_i\right)\right]\\
                                                  E_i\circ\xi_i \\
                                                \end{array}
                                              \right],
          \\
            P_i(q_i,\dot{q}_i,s_i)\Theta_i =&  M_iN_i\mathcal{M}_i\left(q_i\right)s_i+N_iC_i\left(q_i,\dot{q}_i\right)s_i\\
            &-N_i\dot{\mathcal{M}}_i\left(q_i\right)s_i +N_iY_i\left(q_i,\dot{q}_i,{\myr \dot{q}_{ri}, \ddot{q}_{ri}}\right)\Theta_i, \end{align*}
 with
\begin{align*}
             Q_i(q_i,\dot{q}_i,s_i)\Theta_i=& A_iN_i\mathcal{M}_i\left(q_i\right)s_i-Y_i\left(q_i,\dot{q}_i,{\myr \dot{q}_{ri}, \ddot{q}_{ri}}\right)\Theta_i,\\
             L_i(q_i,s_i)\Theta_i=&\Mathcal{M}_i\left(q_i\right)s_i.
 \end{align*}
Now, we are ready to construct the control law as follows
 \begin{align*}
  \tau_i&=-K_{i}s_i-\rho_i(q_i,\dot{q}_i,s_i,\zeta_i)\hat{\omega}_i+A_i\xi_i,\\
  \dot{\xi}_i&=M_i\xi_i+N_i\tau_i,\\
  \dot{\hat{\omega}}_i&=\Lambda^{-1}\rho_{i}^T(q_i,\dot{q}_i,s_i,\zeta_i)s_i,\\
  \dot{\zeta}_i&=M_i\zeta_i+P_i(q_i,\dot{q}_i,s_i),\\
   \dot{S}_i&= \mu_1\sum\nolimits_{j=1}^{5}{a_{ij}(S_j-S_i)},\\
\dot{\eta}_i &=S_i\eta_i + \mu_2\sum\nolimits_{j=1}^{5} {a_{ij}(\eta_j-\eta_i)}.
 \end{align*}
Select the following parameters: $\mu_1=\mu_2=2$, $K_i=40I_2$, $\alpha=6$, $\Lambda_i=0.15I_{17}$. The actual values of $\Theta_i$, $\psi_{i}$, $\sigma_{i}$ and $\phi_{ik}$ are given as: 
\begin{align*}
  \Theta_1 =& \col(0.64,1.10, 0.08,0.64,0.32),~\psi_1=\col(6,8), \\
   \Theta_2 =&\col(0.76,1.17,0.14,0.93,0.44), ~\psi_2=\col(-1,-2),\\
  \Theta_3 =& \col(0.91,1.26,0.22, 1.27,0.58),~\psi_3=\col(-2,-5), \\
  \Theta_4 =& \col(1.10,1.36,0.32,1.67,0.73), ~\psi_4=\col(3,5),\\
  \Theta_5 =& \col(1.21,1.16,0.12,1.45, 1.03),~ \psi_5=\col(-3,-2.5),
\end{align*}
$\sigma_{i}=\col(0.1,0.2)$ and $\phi_{ik}=0$. The simulation is conducted with the following initial conditions: $q_{i}=0$, $\hat{\Theta}_i=0$, $\zeta_i=0$, $\hat{\omega}_i=0$,  $\xi_i=0$, $\forall i$, and
\begin{align*}
S_1\left(0\right)&=\left[\begin{array}{cc}0 & 3 \\ -6 & 0 \\\end{array}\right],~~S_2\left(0\right)=\left[
                     \begin{array}{cc}
                       0 & -2 \\
                      1 & 0 \\
                     \end{array}
                   \right],\nonumber\\
S_3\left(0\right)&=\left[
                     \begin{array}{cc}
                       0 & -2 \\
                       -3 & 0 \\
                     \end{array}
                   \right],~~S_4\left(0\right)=\left[
                     \begin{array}{cc}
                       0 & -2 \\
                       -3 & 0 \\
                     \end{array}
                   \right],\nonumber\\
  S_5\left(0\right)&=\left[
                     \begin{array}{cc}
                       0 & 2 \\
                       -3 & 0 \\
                     \end{array}
                   \right],~\eta_1\left(0\right)=\col(0.2, 0.5),\nonumber\\
  \eta_2\left(0\right)&=\col(-0.6,0.3),~\eta_3\left(0\right)=\col(-0.1, 0.4),\nonumber\\
 \eta_4\left(0\right)&=\col(-0.6,0.6),~\eta_5\left(0\right)=\col( 0.9,0.2).\nonumber
\end{align*}
The errors in Figure \ref{fig2} show that consensus of both $q_i$ and $\dot{q}_i$ is achieved among all the five agents. The trajectories of $\eta_i$ and $S_i$ in Figure \ref{fig3} show that all five agents converge to {\myr an autonomous system arising from the communication network, the inherent properties and the initial states of the agents}.
\section{Conclusion}\label{section6}

This paper proposed a novel design for leaderless consensus and disturbance rejection problem of multiple Euler-Lagrange agents. In this setting, all agents must converge to a common behavior while being affected by persistent disturbances with unknown biases, amplitudes, initial phases and frequencies. The main feature of the proposed design is that none of the agents has information of a common reference model or of a common reference trajectory. Rather, all agents collaborate with each other through a communication network to achieve a common reference trajectory, and simultaneously reject persistent disturbances.
The analysis shows that the generalized coordinates and velocities of the multiple Euler-Lagrange systems converge to common time-varying states in a distributed way.

\footnotesize
\bibliographystyle{ieeetr}

\end{document}